\documentclass[a4paper]{amsart}

\usepackage{graphicx} % Required for inserting images
\usepackage{amsthm}
\usepackage{tikz-cd}
\usepackage{amscd}
\usepackage{hyperref}
\usepackage{mathrsfs}
\usepackage{xcolor}
\usepackage{setspace}
\usepackage{csquotes}
\usepackage[only, llbracket, rrbracket]{stmaryrd}
\usepackage{amssymb, amsmath, amsthm, mathrsfs, amsfonts, latexsym, bbm}
\usepackage{comment}
\usepackage[shortlabels]{enumitem}

\newtheorem{theorem}{Theorem}[section]
\newtheorem{corollary}[theorem]{Corollary}
\newtheorem{lemma}[theorem]{Lemma}
\newtheorem{proposition}[theorem]{Proposition}

\newtheorem{defi}[theorem]{Definition}
\newtheorem{example}[theorem]{Example}

\newcounter{claim}
\newcommand{\norm}[1]{\left\lVert#1\right\rVert}
\newcommand{\abs}[1]{\left\lvert#1\right\rvert}

\newcommand{\supp}{\mathrm{Supp}}

\newcommand{\St}{\mathrm{St}}

\newcommand{\Clop}{\mathrm{Clop}}
\newcommand{\Exh}{\mathrm{Exh}}
\newcommand{\mI}{\mathcal{I}}
\newcommand{\nwd}{\mathrm{nwd}}
\newcommand{\FIN}{\mathrm{Fin}}

\title{Grothendieck ideals of $\ell_\infty$}
\author{M. Hru\v s\' ak, M. A. Rincón-Villamizar, L. S\'aenz and C. Uzcátegui Aylwin}
\date{July 2026}

\thanks{
{The first and third author gratefully acknowledge support received from a DGAPA-PAPIIT grant IN107526 and  a SECIHTI grant CBF-2025-I-898. The second and fourth authors thank for the support received from the Universidad Industrial de Santander.
 }}

\address{Centro de Ciencas Matem\'aticas\\
UNAM\\
A.P. 61-3, Xangari, Morelia, Michoac\'an\\
58089, M\'exico}

\email{michael@matmor.unam.mx, luisdavidr@ciencias.unam.mx}

\address{Universidad Industrial de Santander, Bucaramanga, Colombia}

\email{marinvil@uis.edu.co, cuzcatea@saber.uis.edu.co}

\begin{document}

\begin{abstract} We answer several questions in the literature concerning the\break Grothendieck property of ideals of the Banach lattice $\ell_\infty$ that contain $c_0$. Any such ideal can be represented as a space $c_{0,\mathcal I}$ for $\mathcal I$ an ideal over the natural numbers.  We provide a characterization of when $c_{0,\mathcal I}$ is a Grothendieck space in terms of finitely additive measures over $\mathcal P(\omega)$ and elements of $\mathcal I$. Using this characterization we show that there are analytic ideals $\mathcal I$ such that $c_{0,\mathcal I}$ is Grothendieck. In the opposite direction we show that for any AD family $\mathcal A$, $c_{0,\mathcal I(\mathcal A)}$ and $C(K_\mathcal A)$ are not Grothendieck spaces, and that for most of the Borel ideals present in the literature, $c_{0,\mathcal I}$ is not Grothendieck. In particular, the family of ideals that do not have the Grothendieck property is cofinal in the Rudin-Keisler order, so the Grothendieck property is not downward closed in the Kat\v{e}tov order. Continuing the work in \cite{Sobota-Zuchowski, Zuchowski}, we also provide similar results on the Nikodym property of the Boolean subalgebras of $\mathcal P(\omega)$ generated by the ideal $\mathcal I$.
\end{abstract}

\maketitle

In this note, we continue the analysis on the structure of the ideals of $\ell_\infty$, the Banach lattice of all bounded sequences of reals equipped with the supremum norm and pointwise order, previously developed on \cite{Complementedideals,Uzcategui-complemented,Uzcategui-I-convergence}.  For the following analysis, it is important to distinguish between the two kinds of ideal with which we deal. 

Recall that a \emph{Banach lattice} is a Banach space $X$ equipped with a lattice order relation such that
\begin{enumerate}
    \item[$\bullet$] if $x\leq y$ and $z\in X$, then $x+z\leq y+z$;
    \item[$\bullet$] if ${\bf 0}\leq x$ and $a\in[0,\infty)$, then ${\bf 0}\leq ax$;
    \item[$\bullet$] if $\abs{x}=x\lor-x$, then for any $x,y\in X$, if $\abs{x}\leq \abs{y}$ then $\norm{x}\leq \norm{y}$.
\end{enumerate}
An \emph{ideal} $Y$ of a Banach lattice $X$ is a linear subspace of $X$, such that if $y\in Y$, $x\in X$ and $\abs{x}\leq \abs{y}$, then $x\in Y$. 

On the other hand, given a nonempty set $X$, $\mathcal I\subseteq \mathcal P (X)$ is an \emph{ideal} if $A,B\in \mathcal I$ implies that $A\cup B\in \mathcal I$, and if $A\in \mathcal I$ and $B\subseteq A$, then $B\in \mathcal I$. An ideal is \emph{proper} if $X\notin \mathcal I$. All ideals we consider consist of subsets of a countable set (called its base set), are proper, and contain all finite subsets of the base set. For ease of notation, we usually assume that all ideals have as their base set $\omega$ (the set of all natural numbers).

 Given an ideal $\mathcal I\subseteq \mathcal P (\omega)$ and a sequence $(x_n)_{n\in \omega}\in \ell_\infty$ we say that its $\mathcal I$-limit is $r\in \mathbb R$ if for any $\epsilon>0$, $\{n\in \omega: \abs{x_n-r}\ge \epsilon\}\in \mathcal I$, and denote this fact by $\mathcal I \text{-}\lim x_n=r$. We denote

$$c_{0,\mathcal I}=\{(x_n)_{n\in \omega}\in \ell_\infty: \mathcal I\text{-}\lim x_n=0\},$$
and equip this space with the norm inherited from $\ell_\infty$. In \cite{Uzcategui-I-convergence} it is noted that $X\subseteq \ell_\infty$ is a closed ideal of $\ell_\infty$ that contains $c_0$ if and only if $X=c_{0,\mathcal I}$ for some ideal $\mathcal I$. 

Since Diestel's influential work \cite{DIESTEL197397}, Grothendieck spaces have caught the attention of functional analysts. The recent survey \cite{GrothendieckLandscape} presents  the current state of the art in these spaces and contains a list of questions relevant to the development of the theory.

With the notation just introduced, Question (9) of \cite{GrothendieckLandscape} can be stated as: \emph{When is $c_{0,\mathcal I}$ a Grothendieck space?} Theorem \ref{Main Grothendieck} provides several characterizations that answer this question. 

This result, together with the main result of \cite{Complementedideals}, settles the question implicitly asked by Kania in \cite{Kania}: \emph{Is it true that if $c_{0,\mathcal I}$ does not have a complemented copy of $c_0$, then $c_{0,\mathcal I}$ is complemented in $\ell_\infty$?}

This last question has an interesting story which we now briefly explain.
Recall that  $\ell_\infty$ is a Grothendieck space, as was proved by Grothendieck himself in \cite[Théorème 9]{GrothendieckSurLA}, also recall that the Grothendieck property is preserved under  surjective linear operators and that $c_0$, obviously, is not a Grothendieck space. 

The previous paragraph is a short ``unnecessarily high-tech" proof of Phillips-Sobczyk theorem stating that $c_0$ is not complemented in $\ell_\infty$ (see \cite[Theorem 2.5.5]{Kalton} for an elementary proof). An obvious implication of the argument presented in the previous paragraph is that if $c_{0,\mathcal I}$ is complemented in $\ell_\infty$, then it does not have a complemented copy of $c_0$. Kania's question asks wether the converse is true. Theorem \ref{nowhere dense} answers this question in the negative.
Actually, it has been known for several years (see \cite{Leonetti}) that for any analytic ideal $\mathcal I$, $c_{0,\mathcal I}$ is not complemented in $\ell_\infty$. So, we find it interesting that there are analytic ideals such that $c_{0,\mathcal I}$ is Grothendieck.

Another line of research deals with the Nikodym property of Boolean algebras defined by ideals, a topic developed in \cite{Sobota-Zuchowski, Zuchowski}. We also present some results regarding this notion that use techniques similar to the ones used to analyze the Grothendieck property.

\section{Preliminaries}

We start with the main definitions. Any unexplained notions will be subsequently developed.

\begin{defi}
    Let $X$ be a Banach space. $X$ is \emph{Grothendieck} (or has the Grothendieck property) if any weak$^*$ null sequence in $X^*$ is weakly null. We say $\mathcal I$ is \emph{Grothendieck} if $c_{0,\mathcal I}$ is Grothendieck.
\end{defi}

It is well known that a Banach space $X$ is Grothendieck if and only if any bounded operator $T:X\to c_0$ is weakly compact; see \cite{GrothendieckLandscape} for more equivalences.

The following result appears in \cite[Corollary 2]{Cembranos-Grothendieck}:

\begin{corollary}\label{Grothendieck for C(K)}
    Let $X$ be a Banach space isomorphic to a $C(K)$ space. $X$ has the Grothendieck property if and only if $X$ does not have complemented copies of $c_0$. In particular, $K$ cannot contain non-trivial convergent sequences.
\end{corollary}

Since $c_{0,\mathcal I}$ spaces are $C(K)$ spaces, as we will shortly explain, this implies that for any ideal $\mathcal I$, $c_{0,\mathcal I}$ has the Grothendieck property if and only if it does not have complemented copies of $c_0$. Furthermore, for any non-tall ideal $\mathcal I$, $c_{0,\mathcal I}$ is not Grothendieck.

\begin{defi}
    Let $\mathbb B$ be a Boolean algebra. $\mathbb B$ has the \emph{Nikodym} property if for any sequence of measures $(\mu_n)_{n\in \omega}\subseteq \text{ba}(\mathbb B)$ such that for any $B\in \mathbb B$, $(\mu_n(B))_{n\in \omega}$ is bounded implies  $(\mu_n)_{n\in \omega}$ is norm bounded. We say $\mathcal I$ is \emph{Nikodym} if the algebra $\mathcal I\cup \mathcal I^*$ is Nikodym.\footnote{Notice this definition defers from the one used in \cite{Zuchowski}.}
\end{defi}

It is a standard fact that $\mathbb B$ has the Nikodym property if and only if any sequence  $(\mu_n)_{n\in \omega}\subseteq \text{ba}(\mathbb B)$, such that for any $B\in \mathbb B$, $\lim _{n\in \omega}(\mu_n(B))=0$, is a weak$^*$ null sequence; see \cite[Proposition 2.4]{Sobota-Zdomskyy}.

The following are standard notions in their respective areas.  

\smallskip

We consider $\mathcal{P}(\mathbb{\omega})$ equipped with the natural topology inherited from the product topology of $2^\omega$ by means of the characteristic functions. Whenever we talk about a subset of $\mathcal{P}(\omega)$ having any topological property: closed, Borel, analytic, etc., we refer to this topology. Recall that an ideal is \textit{tall} if for any infinite $A\subseteq \omega$ there is an infinite $I\in \mathcal I$ such that $I\subseteq A$.

Given an ideal, we denote $\mathcal I^+=\mathcal P(\omega)\setminus \mathcal I$. For any $X\in \mathcal I^+$, we denote $\mathcal I\lvert_X=\{A\cap X:A\in \mathcal I \}.$
There are several orders for ideals on $\omega$. \textit{The Kat\v{e}tov order}: given two ideals $\mathcal I$, $\mathcal J$ on $\omega$, $\mathcal I\leq_K \mathcal J$ if there is a map $f:\omega\to \omega$ such that for any $A\in \mathcal I$, $f^{-1}(A)\in \mathcal J$. \textit{The Rudin-Keisler order}: given two ideals $\mathcal I$, $\mathcal J$ on $\omega$, $\mathcal I\leq_{RK} \mathcal J$ if there is a map $f:\omega\to \omega$ such that for any  $A\subseteq \omega$, $A\in \mathcal I$ if and only if $f^{-1}(A)\in \mathcal J$. Clearly, the Kat\v{e}tov order refines the Rudin-Keisler order. The former order has been widely used to study combinatorial properties of ideals, see \cite{Hrusakidealsurvey}. It is easy to see that for any ideal $\mathcal I$, and $X\in \mathcal I^+$, $\mathcal I\leq_{K}\mathcal I\lvert_X$

If $\mI$ is any ideal and $(\mI_n)_{n\in \omega}$ is a sequence of ideals, we define the Frol\'ik sum of  $(\mI_n)_{n\in \omega}$ by $\mI$ as

$$\sum_{\mI} \mI_n=\{A\subseteq \omega^2: \{n:\{m:(n,m)\in A\}\notin \mI_n\}\in \mI\}.$$

It is easy to see that this is an ideal over $\omega^2$. A particular case of this is when for each $n\in \omega$, $\mathcal I_n=\mathcal J$. In such case we modify the notation and write

$$\mI\times \mathcal J=\{A\subseteq \omega^2: \{n:\{m:(n,m)\in A\}\notin \mathcal J\}\in \mI\}.$$

The following four ideals will be central to our analysis:

$$\mathcal Z=\Bigg\{A\subseteq \omega:\lim_{n\to \infty}\frac{\abs{A\cap[2^n,2^{n+1})}}{2^n}=0\Bigg\},$$

 \textit{The ideal of nowhere dense subsets of $\mathbb Q$}:

$$\nwd=\{A\subseteq \mathbb Q: A \text{ is nowhere dense}\}.$$

\textit{The ideal generated by the convergent sequences in $\mathbb Q$}: $$\text{conv}=\{A\subseteq \mathbb Q: A \text{ has finitely many limit points}\},$$

and \textit{Solecki's ideal}:

Denote by $$\Omega=\{A\in \Clop(2^\omega):\lambda(A)=1/2\}$$

Where $\Clop(2^\omega)$ denotes the set of all clopen sets in the Cantor set, and $\lambda$ the Lebesgue measure over $2^\omega$. For any $x\in 2^\omega$, $I_x=\{A\in \Omega:x\in A\}$. Solecki's ideal is the ideal generated by these sets, that is:

$$\mathcal S=\langle\{I_x:x\in 2^\omega\}\rangle.$$

Recall that a function $\varphi:\mathcal P (\omega)\to \mathbb R\cup\{\infty\}$ is a lower semicontinuos (lsco) submeasure if the following hold:

\begin{enumerate}
    \item $\varphi(\varnothing)=0$,
    \item for all $A\subseteq B$, $\varphi(A)\leq \varphi(B)$,
    \item for all $A,B\subseteq \omega$, $\varphi(A\cup B)\leq \varphi(A)+\varphi(B)$,
    \item for all $A\subseteq \omega$, $\varphi(A)=\sup\{\varphi(F):F\in [A]^{<\omega}\}$, and
    \item for all $F\in[\omega]^{<\omega}$, $\varphi(F)<\infty$.
\end{enumerate}

Given a lsco submeasure its exhaustive ideal is defined as $$\Exh(\varphi)=\{A\subseteq \omega:\lim_{n\to \infty}\varphi(A\setminus n)=0\}.$$ While its ideal of sets with finite submeasure is defined as $$\FIN(\varphi)=\{A\subseteq \omega:\varphi(A)<\infty\}.$$ It is immediate that $\Exh(\varphi)\subseteq \FIN(\varphi)$ and that $\FIN(\varphi)$ is an $F_\sigma$ ideal.

A lsco submeasure is called non-pathological if there exists a sequence $(\lambda_n)_{n\in \omega}$ of non-negative $\sigma$-additive measures over $\omega$ such that for any $A\subseteq \omega$, $\varphi(A)=\sup_{n\in \omega}\lambda_n(A)$. By the main result of \cite{SOLECKI199951} $\mathcal I$ is an analytic $P$-ideal if and only if there exists a lsco submeasure on $\omega$ such that $\mathcal I=\Exh(\varphi)$. By the measure dichotomy theorem of \cite{HrusakBorel}, an analytic $P$-ideal admits a non-pathological lsco submeasure if and only if for any $X\in \mathcal I^+$, $\mathcal I\lvert_X\leq_K \mathcal Z$.

 It is easy to see  that there are no $F_\sigma$ (proper) ideals Kat\v{e}tov above $\text{conv}$, see \cite[Corollary 3.14]{HrusakRamsey}. Therefore, if $\mathcal I$ is a $P$-analytic ideal such that for any $X\in \mathcal I^+$, $\text{conv}\leq \mathcal I\leq \mathcal I\lvert_X\leq \mathcal Z$, there is a non-pathological lsco submeasure $\varphi$ such that $\mathcal I=\Exh(\varphi)$ and $\varphi(\omega)<\infty$. 

In contrast, given a mapping $f:\omega\to \omega$ such that $\sum_{n\in \omega}f(n)=\infty$, the summable ideal associated to $f$ is defined as: 

$$\mathcal I_f=\{A\subseteq \omega:\sum_{n\in A}f(n)<\infty\}.$$

Notice that $\varphi_f(A)=\sum_{n\in A}f(n)$ defines a $\sigma$-additive infinite measure on $\omega$ such that $\mathcal I_f=\Exh(\varphi_f)=\FIN(\varphi_f)$. Therefore, the family of summable ideals are examples of analytic $P$-ideals such that they admit a lsco submeasure with $\varphi(\omega)=\infty$. In particular, they are not Kat\v{e}tov above $\text{conv}$, furthermore, any summable ideal is Kat\v{e}tov below $\mathcal Z$.

The ideals $\FIN^{\alpha}$, for $\alpha<\omega_1$ are recursively defined as follows. 
\begin{enumerate}
    \item For $\alpha=0$, $X_{0}=\{0\}$ $\FIN^{0}=\{\emptyset\}$.
    \item For $\alpha=1$, $X_1=\omega$ and $\FIN=[\omega]^{<\omega}$.
    \item For $\alpha=\alpha'+1$ successor ordinal, $X_{\alpha}=\omega\times X_{\alpha'}$ and $$\FIN^ {\alpha}=\{A\subseteq X_{\alpha}: \{n: \{x:(n,x)\in A\}\notin \FIN^{\alpha'}\}\in \FIN\}.$$
    \item For $\alpha$ limit,  define $X_{\alpha}=\bigcup_{\beta<\alpha}X_{\beta}$ and 
$$\FIN^{\alpha}=\{A\subseteq X^{\alpha}: \{\beta:\{x:(\beta,x)\in A\}\notin \FIN^{\beta}\} \text{ is bounded in } \alpha\}\}$$
\end{enumerate}
Given a Boolean algebra $\mathbb B$, we denote by $\St(\mathbb B)$ its Stone space, that is the space of all ultrafilters of the algebra $\mathbb B$. For any $B\in \mathbb B$ we denote $\overline{B}=\{x\in \St(\mathbb B): B\in x\}$.

Recall that a family $\mathcal A\subseteq [\omega]^{\omega}$ is called almost disjoint (AD) if for any $A,B\in \mathcal A$, $A\cap B$ is finite. Given an AD family $\mathcal A$ we denote by $\mathcal I(\mathcal A)$ the ideal generated by the family, that is, $$\mathcal I(\mathcal A)=\Bigg\{B\subseteq \omega:\exists A_0,\dots, A_n\in \mathcal A, \abs{B\setminus \bigcup_{i\leq n} A_i}<\omega\Bigg\}$$

For any AD family $\mathcal A$ we define its \emph{Franklin compactum} as $K_\mathcal A=\St(\mathbb B(\mathcal A) )$, where $\mathbb B(\mathcal A)$ is the Boolean algebra generated by $\mathcal A\cup [\omega]^{<\omega}$. This space can also be presented as the one point compactification of the \emph{Mr\'oka-Isbell space}, $\psi(\mathcal A)$.
 
Given an ideal $\mathcal I$, we denote, as usual, its dual filter by $\mathcal I ^*$; i.e., $$\mathcal I^*=\{\omega\setminus I:I\in \mathcal I\}.$$ We denote $K_\mathcal I=\{\mathcal U\in \beta\omega:\mathcal I^*\subseteq \mathcal U\}$ the corresponding compact subspace of $\beta\omega$, and $U_\mathcal I=\omega\sqcup \{\mathcal U\in \omega^*: \mathcal U \cap \mathcal I\neq \varnothing\}$. $U_\mathcal I$ is clearly an open subset of $\beta\omega$ and defines a locally compact space. We denote $U_\mathcal I^*$ its one point compactification and $p_\mathcal I$ its point at infinity. Notice that $U_\mathcal I^*$ has an alternative presentation.  Denote $\mathbb{B}(\mathcal I)=\mathcal I\cup \mathcal I^*$, the Boolean algebra defined by $\mathcal I$, then $U_\mathcal I^*\simeq\St(\mathbb B(\mathcal I))$.

We denote
$$c_{\mathcal I}=\{(x_n)_{n\in \omega}\in \ell_\infty: \exists r\in \mathbb R \text{ such that }\mathcal I\text{-}\lim x_n=r\},$$

equipped with the supremum norm.

Given a compact space $K$, we denote by $C(K)$ the space of continuous functions over $f$ with the supremum norm. Given $U$ a locally compact space, we denote by $C_0(U^*)$ the space of continuous functions over the one point compactification $U^*$ that vanish at the point at infinite. 

It is easy to see that for any ideal $\mathcal I$, $C_0(U_\mathcal I^*)$, $ c_{0,\mathcal I}$, $ c_{\mathcal I}$, and $ C(U_\mathcal I^*)$ are isomorphic as Banach spaces, although not as Banach lattices.
It is also easy to see that for any  AD family $\mathcal A$, $C_0(K_\mathcal A)\subseteq C_0(U_{\mathcal I(\mathcal A)}^*)$, and that if we consider the closed linear subspace of $\ell_\infty$ spanned by $c_0$ together with $\{\chi_A:A\in \mathcal A\}$, then this space is isomorphic to $C_0(K_\mathcal A)$, furthermore, $C_0(K_\mathcal A)$ and $ C(K_\mathcal A)$ are isomorphic as Banach spaces.

Given a topological space $(X,\tau)$ we denote by $\text{Borel(X)}$ the family of Borel sets of $X$, that is, the $\sigma$-algebra generated by $\tau$.

Given a compact space $K$, denote by $\mathcal M(K)$ the space of Radon (signed $\sigma$-additive) measures equipped with the norm $$\norm\mu=\sup\{\abs{\mu(A)}+\abs{\mu(B)}:A,B\in \text{Borel}(K),A\cap B= \varnothing\}.$$ Recall that the \emph{Riesz representation theorem} states that $\mathcal M(K)$ is isometrically isomorphic to $C(K)^*$, the dual space of $C(K)$.

Given a Boolean algebra $\mathbb B$, $\mu:\mathbb B\to \mathbb R$ is a (signed) measure if it is finitely additive, $\mu(\varnothing)=0$, and bounded. Denote by $\text{ba}(\mathbb B)$ the set of all measures over $\mathbb B$, equipped with the norm of \emph{total variation}, i.e., $$\norm{\mu}=\sup\{\abs{\mu(A)}+\abs{\mu(B)}:A,B\in \mathbb B, A\land B=0_\mathbb B\}.$$ It is a standard fact that any $\mu \in \text{ba}(\mathbb B)$ has a unique extension to a Radon measure $\hat\mu:\text{Borel}(\St(\mathbb B))\to \mathbb R$, regular with respect to $\mathbb B\simeq \Clop(\St(\mathbb B))$, i.e., for any $B\in \text{Borel}(\St(\mathbb B))$, $\hat\mu(B)=\sup\{\mu(A):A^*\subseteq B\}$, where $A^*=\{x\in \St(\mathbb B):A\in x\}$. Furthermore, $\norm{\mu}=\norm{\hat\mu}$, so $\mathcal M(\St(\mathbb B))$ is isometrically isomorphic to $\text{ba}(\mathbb B)$.

For an ultrafilter $\mathcal U$ on $\omega$ we denote $\delta_\mathcal U$ its Dirac measure; the finitely additive measure on $\omega$ such that $\delta_{\mathcal U}(A)=1$ if and only if $A\in \mathcal U$.

\section{Grothendieck ideals}

Before proceeding we require a definition.

\begin{defi}
    Let $\mathcal I$ be an ideal. A sequence of  elements of $\mathcal I$, $(I_n)_{n\in \omega}$, 
    and a sequence of finitely additive measures over $\mathcal P(\omega)$ of norm bounded by one, $(\mu_n)_{n\in \omega}$, are an anti-Grothendieck pair withnessed by $\epsilon>0$  if 
    
    \begin{enumerate}
        \item For any $I\in \mathcal I$, $\lim_{n\to \infty}\mu_n(I)=0$, and
        \item For any $n\in \omega$, $\abs{\mu_n(I_n)}>\epsilon$.
    \end{enumerate}

\end{defi}

The following result is our main characterization. In what follows we will mainly use the equivalence between (1) and (4).

\begin{theorem}\label{Main Grothendieck}
    Let $\mI$ be an ideal. The following conditions are equivalent:
    
    \begin{enumerate}
        \item $c_{0,\mI}$ does not have the Grothendieck property.
\item There is a an anti-Grothendieck pair $(I_n)_{n\in \omega}$, $(\mu_n)_{n\in \omega}$. 
        
\item There is an anti-Grothendieck pair $(I_n)_{n\in \omega}$, $(\mu_n)_{n\in \omega}$, such that the elements of $(I_n)_{n\in \omega}$ are pairwise disjoint. 

\item There is an anti-Grothendieck pair $(I_n)_{n\in \omega}$, $(\mu_n)_{n\in \omega}$ witnessed by 1, such that the elements of $(I_n)_{n\in \omega}$ are pairwise disjoint, the elements of $(\mu_n)_{n\in \omega}$ are non-negative, normalized, and for any $n\in \omega$, $\supp(\mu_n)\subseteq I_n$.

    \item There is a sequence of finitely additive non-negative measures over $\omega$ of norm bounded by 1, $(\mu_{n})_{n\in \omega}$, a sequence $(I_n)_{n\in \omega}\subseteq \mathcal I$ and a bounded sequence of reals $(a_n)_{n\in \omega}$, such that:
    \begin{enumerate}
    \item For any $I\in \mathcal I$, $\lim_{n\to\infty}\mu_n(I)=0$
        \item For any $n\in \omega$, $a_n\mu_n(I_n)=1$.
        
    \end{enumerate}

    \end{enumerate} 
\end{theorem}
\begin{proof}
    $(1)\to (3)$. Let us work in $C(U_\mathcal I^*)$, as it is isomorphic to $c_{0,\mathcal I}$, it cannot be Grothendieck.  Then there is a $w^*$-null sequence $(\nu_n)_{n\in \omega}\subseteq B_{C(U_\mathcal I^*)^*}$ that is not $w$-null. Since it is not $w$-null, $\{\nu_n:n\in \omega\}$ is not relatively weakly compact.  By the Riesz representation theorem we can represent $(\nu_n)_{n\in \omega}$ as a bounded sequence in $M(U_\mathcal I^*)$. By  Grothendieck's characterization of relatively weakly compact subsets of spaces of measures \cite[Théorème 2]{GrothendieckSurLA} (See also \cite[Theorem 5.3.2]{Kalton}) we can find a subsequence of $(\nu_n)_{n\in \omega}$, a sequence of open disjoint sets $(U_n)_{n\in \omega}$, and $\epsilon>0$ such that after re indexing we get that for any $n\in \omega$, $\abs{\nu_n(U_n)}>2\epsilon$. Using regularity we may find for any $n\in \omega$ $I_n\in \mathbb B(\mathcal I)$, such that $\overline{I_n}\subseteq U_n$ and $\abs{\nu_n(\overline{I_n})}>\epsilon$. Since the sequence $(I_n)_{n\in \omega}$ must be disjoint, there is at most one $m\in \omega$ such that $I_m\in \mathcal I^*$. So we may re index once again so that for any $n\in \omega$, $I_n\in \mathcal I$.
For any $n\in \omega$ extended $\nu_n$ to $\mu_n:\mathcal P(\omega)\to[-1,1]$. Since for any $I\in \mathcal I$, $\chi_I\in C(U_\mathcal I^*)$, hence $\lim_{n\to \infty}\mu_n(I)=\lim_{n\to \infty}\nu_n(I)=0$.

$(3)\to (4)$. Fix $n\in \omega$, consider $\hat{\mu}_n:\text{Borel}(\beta \omega)\to [-1,1]$. By Hanh's decomposition there are $A_n^+,A_n^-\in\text{Borel}(\beta\omega))$ such that $A_n^+\sqcup A_n^-=\beta\omega$, $A_n^+$ is positive and $A_n^-$ is negative. Use regularity to find $B_n^+,B_n^-\subseteq \omega$ such that $\overline{B_n^+}\subseteq A^+$, $\overline{B_n^-}\subseteq A^-$, $\mu_n(B_n^+)>\hat \mu_n(A_n^+)-\epsilon/6$, $\mu_n(B_n^-)<\hat \mu_n(A_n^-)+\epsilon/6$. 

\textbf{Claim} For any $C\subseteq \omega$, $\abs{ \mu_n(C\setminus(B_n^+\cup B_n^-) )}<\epsilon/3$.
\,

Pick $C\subseteq \omega$, 

$$\abs{ \mu_n(C\setminus(B_n^+\cup B_n^-) )}=$$
$$\abs{\hat\mu_n\Big(\Big(\overline{C}\cap \Big(A_n^+\sqcup A_n^-\Big)\Big)\setminus\Big( \overline{B_n^+}\cup \overline{B_n^-}\Big)\Big)}=$$
$$\abs{\hat\mu_n\bigg(\Big(\Big(\overline{C}\cap A_n^+\Big)\setminus \overline{B_n^+}\Big)\sqcup \Big( \Big(\overline{C}\cap A_n^-\Big)\setminus \overline{B_n^-}\Big)\bigg)}=$$
$$\abs{\hat\mu_n\Big(\overline{C}\cap A_n^+\Big)-\mu_n\Big(C\cap B_n^+\Big)+\hat\mu_n\Big(\overline{C}\cap A_n^-\Big)-\mu_n\Big(C\cap B_n^-\Big)}\leq$$
$$\abs{\hat\mu_n\Big(\overline{C}\cap A_n^+\Big)-\mu_n\Big(C\cap B_n^+\Big)}+\abs{\hat\mu_n\Big(\overline{C}\cap A_n^-\Big)-\mu_n\Big(C\cap B_n^-\Big)}<\epsilon/3$$

This computation establishes the claim.

As $\epsilon<\abs{\mu_n(I_n)}\leq \abs{\mu_n(I_n\cap B^+)}+\abs{\mu_n(I_n\cap B^-)}+\abs{ \mu_n(I_n\setminus(B^+\cup B^-) )}$,
either $\abs{\mu_n(I_n\cap B^+)}>\epsilon/3$ or $\abs{\mu_n(I_n\cap B^-)}>\epsilon/3$. Define $J_n\subseteq I_n$ as $J_n=I_n\cap B^+$ or $J_n=I_n\cap B^-$ making sure that
$\abs{\mu_n(J_n)}>\epsilon/3$. Further, define $\nu'_n=\mu\lvert_{J_n}$ if $J_n=I_n\cap B^+$ or $\nu'_n=-\mu\lvert_{J_n}$ in the other case. In either case $\nu'_n$ is a nonnegative measure such that $\nu'(J_n)>\epsilon/3$ and the sequence $(\nu_n)_{n\in \omega}$ has disjoint supports. Define for any $n\in \omega$, $\nu_n=\frac{\nu_n'}{\nu_n'(J_n)}$. Then obviously $\norm{\nu_n}=1$ and $\nu_n(J_n)=1$.

We are left with showing that for any $I\in \mathcal I$, $\lim_{n\to \infty}\nu_n(I)=0$. Fix your favorite $I\in \mathcal I$ and consider $J=I\cap (\bigcup_{n\in \omega}J_n)\in \mathcal I$. Then for any $n\in \omega$, $\nu_n(I)=\abs{\mu_n(J)}/\abs{\mu_n(J_n)}$, so $$\lim_{n\to \infty} \nu_n(I)=\lim_{n\to \infty}\frac{\abs{\mu_n(J)}}{\abs{\mu_n(J_n)}}\leq  \lim_{n\to \infty}\frac{\abs{\mu_n(J)}}{\epsilon/3}=0.$$

$(4)\to (5)$ and $(2)\to (5)$.  Consider $a_n=1/\mu_n(I_n)$.

$(5)\to (1)$. Consider the operator $T:c_{0,\mathcal I}\to c_0$ defined by $T(x)(n)=\int_\omega x d\mu_n$. It is clear that for any $n\in \omega$, $a_n\chi_{I_n}\in c_{0,\mathcal I}$. Therefore  $T$ is surjective, in particular it is not weakly compact. Therefore, $c_{0,\mathcal I}$ is not Grothendieck.

$(4)\to (2)$ is trivial.

\end{proof}

This result implies that for spaces of the form $c_{0,\mathcal I}$ having the Grothendieck, having the positive Grothendieck property, and having the weak Grothendieck property are all equivalent. \begin{footnote}{Recall that a Banach lattice $X$, 1) has the \textit{positive Grothendieck property} if any weakly$^*$ null sequence $(x^{*}_{n})_{n\in \omega}$ of \textit{positive} functionals is also weakly null, 2) has the \textit{weak Grothendieck property} if any weakly$^*$ null sequence $(x^{*}_{n})_{n\in \omega}$ of functionals with \textit{disjoint} terms is also weakly null. See \cite{Wnuk} and \cite{Machrafi} for the respective definitions and some results regarding these properties.} Note that we cannot present the same result for spaces of the form $c_{\mathcal I}$ since it is known that the space of convergent sequence $c$ has the positive Grothendieck property but does not have the weak Grothendieck property nor, obviously, the Grothendieck property. 
    
\end{footnote} 

The reader may wonder if we could relax the hypothesis of the theorem by just assuming there is a pair $(X_n)_{n\in \omega}$ of subsets of $\omega$ and a sequence of measures $(\mu_n)_{n\in \omega}$ that jointly fulfill conditions (1) and (2) in the definition of an anti-Grothendieck pair. The following example shows that we cannot relax the hypothesis in such a way.

\begin{example}
Consider a discrete sequence of ultrafilters. $(\mathcal U_n)_{n\in \omega}$, and $\mathcal I=\bigcap_{n\in \omega}\mathcal U_n^*$. Define $\mu_n=\delta_{\mathcal U_n}$. Since the sequence is discrete, we may find for any $n\in \omega$, $X_n\in \mathcal U_n$ such that for any $m\neq n$, $X_m\notin \mathcal U_m$. It is clear that for any $n\in \omega$, $X_n\in \mathcal I^+$. It is also clear that for any $I \in \mathcal I$
, $\lim_{n\in \omega}\mu_n(I)=0$. But, since $\mathcal P(\omega)/\mathcal I$ is $\sigma$-centered, by the main theorem of \cite{Complementedideals}, $c_{0,\mathcal I}$ is complemented in $\ell_\infty$, so, in particular, it is Grothendieck.
\end{example}

The following results show that there are Borel Grothendieck ideals.

\begin{lemma}\label{nowhere dense}
Let $(I_n)_{n\in\omega}$ be a pairwise disjoint family of nowhere dense subsets of $\mathbb Q$ and let $(\mu_n)_{n\in\omega}$ be a sequence of finitely additive probability measures over $\mathbb Q$ such that $\mu_n(I_n)=1$ for every $n\in\omega$. There is $I\in \nwd$ such that $\lim_{n\to \infty}\mu_n(I)\neq 0$.
\end{lemma}

\begin{proof} Enumerate a basis for open sets of the rationals as $(U_n)_{n\in\omega}$. Given $(I_n)_{n\in\omega}$ and $(\mu_n)_{n\in\omega}$ as in the statement of the lemma we shall recursively construct for every $n\in\omega$
\begin{itemize}
    \item a non-empty open $W_n\subseteq U_n$,
    \item $A_n\in [\omega]^{\omega}$,
    \item $k_n\in A_n$, and 
    \item $J_{k_n}\subseteq I_{k_n} $
\end{itemize}
so that
\begin{enumerate}
    \item $(k_n)_{n\in\omega}$ is strictly increasing,
    \item $(A_n)_{n\in\omega}$ is $\subseteq$-decreasing,
    \item for all $j\in A_n$, $\mu_j (W_n) \leq \frac{1}{2^{n+2}}$, 
    \item $\bigcup_{n\in\omega} W_n \cap \bigcup_{n\in\omega} J_{k_n}=\emptyset$, and
    \item $\mu_{n_k}(J_{k_n})\geq \frac{1}{2}$.
\end{enumerate}

To start, set $A_{-1}=\omega$ and $k_{-1}=-1$. To construct $W_n, A_n, k_n$ and $Y_n$ find a pairwise disjoint family 
$$\{V_i: i< 2^{n+2}\}$$ 
of non-empty open subsets of 
$$U_n\setminus \bigcup_{m<n} J_{k_m}.$$
As all $J_{k_m}$ are nowhere dense and $\mathbb Q$ has no isolated points, this is easy to do.

Now, for each $m\in A_{n-1}$ there is an $i_m<2^{n+2}$ such that
$$\mu_m(V_{i_m})\leq \frac{1}{2^{n+2}}.$$
By the pigeonhole principle, there is an $i<2^{n+2}$ such that the set
$$A_n=\{ m\in A_{n-1}: i_m=i\}$$
is infinite. Let $W_n=V_i$, $k_n=\min A_n\setminus (k_{n-1} +1)$ and 
$$J_{k_n}= I_{k_n}\setminus\bigcup_{m\leq n} W_m .$$

(1)-(4) follow immediately from the definitions of these objects. To see (5), recall that $k_n\in \bigcap_{m\leq n} A_m$, hence $\mu_{n_k} (W_m)\leq \frac{1}{2^{m+2}}$ for every $m\leq n$. So 
$$\mu_{n_k}(\bigcup_{m\leq n} W_m) \leq \sum_{m\leq n} \mu_{n_k}(W_m) \leq \sum_{m\leq n}\frac{1}{2^{m+2}}\leq \frac12,$$
so 
$$\mu (J_{n_k})= \mu_{n_k} (I_{n_k}\setminus \bigcup_{m\leq n} W_m)\geq \frac12.$$
Define $I=\mathbb Q\setminus \bigcup_{n\in\omega} W_n$. $I$ is a nowhere dense subset of $\mathbb Q$ because $\bigcup_{n\in\omega} W_n$ is dense open due to the fact that $W_n$ is a non-empty open subset of $U_n$ for every $n$. Also, for every $n$
$$\mu_{k_n}(I)\geq \mu_{k_n}(J_{n_k})\geq\frac12,$$
by (4) and (5). Hence,  $\lim_{n\to \infty}\mu_n(I)\neq 0$.
\end{proof}

As a direct consequence, we get:

\begin{theorem}\label{little nowhere dense}
    $\nwd $ is a Grothendieck ideal.
\end{theorem}

Since $\nwd$ is a Borel ideal, it is meager. Therefore, $c_{0,\nwd}$ is not complemented in $\ell_\infty$, see \cite{Complementedideals}. Therefore, $c_{0,\mathcal I}$ is not an injective Banach space, in particular, $c_{0,\nwd}\not\simeq \ell_\infty$.

The following easily follows from \cite[Corollary D]{Marciszewski-Sobota}.

\begin{proposition}
    Let $\mathcal I$ be an ideal. There is a sequence of finitely supported positive measures $(\mu_n)_{n\in \omega}$ such that for any $I\in\mathcal I$, $\lim_{n\in \omega}\mu_n(I)=0$ if and only if $\mathcal I\leq_K\mathcal Z$.
\end{proposition}

This implies that any ideal Kat\v{e}tov below $\mathcal Z$ is not Grothendieck.
We can generalize this result in the following way:

\begin{theorem}\label{Result on P}
    Let $\mathcal I$ be a $P$-ideal. $\mathcal I$ is Grothendieck if and only if $\mathcal I\not\leq_{KB} \mathcal Z$.
\end{theorem}
\begin{proof}
    If $\mathcal I\leq_{KB} \mathcal Z$, then by the previous proposition $\mathcal I$ cannot be Grothendieck.

    Assume now that $\mathcal I$ is not Grothendieck. Find an anti-Grothendieck pair such that $(I_n)_{n\in \omega}\subseteq \mathcal I$ is a sequence of pairwise disjoint elements, and $(\mu_{n})_{n\in \omega}$ is a sequence of finitely additive nonnegative measures such that for any $n\in \omega$, $\mu_n(I_n)=1$. The fact that $\mathcal I$ is a $P$-ideal implies that there is $I\in\mathcal I$ such that for any $n\in \omega$, $I\setminus I_n=F_n\in[\omega]^{<\omega}$. $I\in \mathcal I$ implies there is $m\in \omega$ such that for any $n\ge m$, $\mu_n(I)<1/2$. This implies that for any $n\ge m$, $\mu(F_n)\ge 1/2$. Define $\nu_n=\mu_{n+m}\lvert_{F_m}$. Then $(\nu_n)_{n\in \omega}$ is a sequence as in the previous proposition. Therefore $\mathcal I\leq_{KB} \mathcal Z$.
\end{proof}

It is shown in \cite{HrusakBorel} that an analytic $P$-ideal $\mathcal I$ is non-pathological if and only if for any $X\in \mathcal I^+$, $\mathcal I\lvert_X\leq \mathcal Z$. So, we have the following result:

\begin{corollary}
    Let $\mathcal I$ be an analytic $P$-ideal. If $\mathcal I$ is pathological, then there is $X\in \mathcal I^+$ such that $\mathcal I\vert_X$ is Grothendieck.
\end{corollary}

There are proofs of the existence of pathological $F_\sigma$ $P$-ideals in the literature, see \cite{Iliasold} sections 1.8 and 1.9, as well as \cite{Iliasnew} section 4.4. Nevertheless, we could not find an \textit{explicit} example of an ideal with these properties. Thus, we provide the following simple example.

\begin{example}

For every $n\in \omega$, $n>1$, let

$$W_n=\{U\subseteq 2^n: |U|=2^{n-1}\}$$

and $\varphi_n: \mathcal P(W_n)\to \mathbb R$ by

$$\varphi_n(A)=\frac{\min\{|F|: F\subseteq 2^n \land\forall U \in A(  F\cap U\neq\emptyset)\}}{2^{n-1}}.$$

Let $W=\bigcup_{n>1}W_n$ and $\varphi:\mathcal P(W)\to \mathbb R\cup\{\infty\}$ by

$$\varphi(A)= \sum_{n>1} \varphi_n(A\cap W_n).$$

We claim the following holds,
\begin{enumerate}
\item $\varphi$ is a lsco submeasure on $W$.
\item  $\mathcal H=\FIN(\varphi)=\Exh(\varphi)$ is a proper $F_\sigma$ P-ideal on $\bigcup_{n>1}W_n$,
\item  $\mathcal S\leq_\text{K}\mathcal H$, and
\item $\mathcal H$ is pathological and $\mathcal H\nleq_K \mathcal Z$.
\end{enumerate}

Checking (1) and the second equality in (2) are  easy exercises, the rest of the properties of (2) follow from them.  To see (3) note that the function $f:W\to \Omega$ defined by
$f(U)=\bigcup_{s\in U} \langle s\rangle$
is Kat\v etov. Since $\varphi(f^{-1}[I_x])=\sum_{n>1}\frac{1}{2^{n-1}}$, hence $f^{-1}[I_x]\in\mathcal H$ for every $x\in 2^\omega$. (4) Follows from (3) by the measure dichotomy of \cite{HrusakBorel} and the fact that $\mathcal S\nleq_K \mathcal Z$.

\end{example}

In particular, this shows that the natural conjecture that all tall $F_\sigma$ ideals may have isomoporhic $c_{0,\mathcal I}$ spaces is false.

 The Kat\v{e}tov order may be successfully used in other results:
\begin{theorem}\label{Z kills}
   For any ideal $\mathcal I$, if $\mathcal I\leq_{K}\mathcal Z\times \FIN$, then $\mathcal I$ is not Grothendieck.
\end{theorem}

\begin{proof}
    Consider a Kat\v{e}tov reduction $f:\omega^2\to \omega$ from $\mathcal Z\times \FIN$ to $\mathcal I$. If $\mathcal I$ is not tall then it is clear that  it is not Grothendieck. So, assume that $\mathcal I$ is tall.
    
    Let $X_n=f(\{n\}\times \omega)$. If $X_n\in \mathcal I$, let $J_n=X_n$, otherwise use that $\mathcal I$ is tall to find $J_n\in \mathcal I\cap[X_n]^{\omega}$. For any $n\in \omega$, if $J_n$ is infinite, find an ultrafilter such that $J_n\in\mathcal U_n$ and define $\nu_n=\delta_{\mathcal U_n}$. If $J_n$ is finite, define $F_n=\{i\in J_n: f^{-1}(\{i\}) \text{ is infinite}\}\neq \varnothing$ and $\nu_n=\frac{1}{\abs{F_n}}\sum_{i\in F_n}\delta_i$.

    Define $I_m=\bigcup_{n\in [2^{m},2^{m+1})}J_n$ and $\mu_m=\sum_{n\in [2^{m},2^{m})}2^{-m}\nu_n$. Notice that for any $m\in \omega$, $\mu_m(I_m)=1$. 
    
\smallskip
    
    \textbf{Claim} For any $I\in \mathcal I$, $\lim_{m\ \to \infty}\mu_m(I)=0$.

          Pick your favorite $I\in \mathcal I$ and $\epsilon>0$. Consider $$P_I=\{n\in \omega:\abs{\{m:(n,m)\in f^{-1}(I)\}}= \omega\},$$ then $P_I\in \mathcal Z$. Using this fact, find $M\in \omega$ such that for any $m\geq M$, $\frac{\abs{P_I\cap[2^m,2^{m+1})}}{2^m}<\epsilon$.

    For any $m\in \omega$, consider $n\in [2^m,2^{m+1})\setminus P_I$, then $I\cap J_n$ is finite. If $J_n$ is infinite, then $\nu_n(I)=0$. If $J_n$ is finite, then $I\cap F_n=\varnothing$, so $\nu_n(I)=0$. Therefore, for $m\ge M$,

    $$\mu_m(I)=\sum_{i\in P_I\cap[2^m,2^{m+1})}2^{-m}\nu_i(I)+ \sum_{i\in [2^m,2^{m+1})\setminus P_I}2^{-m}\nu_i(I)<\epsilon.$$

    By Theorem \ref{Main Grothendieck}, $\mathcal I$ is not Grothendieck.

\end{proof}

\begin{corollary}
    For any AD family $\mathcal A$, $c_{0,\mathcal I(\mathcal A)}$ and $C(K_\mathcal A)$ are not Grothendieck. 
\end{corollary}

\begin{proof}
    Let $\mathcal A$ be an AD family, it is well known that $\mathcal I(\mathcal A)\leq_K \FIN\times \FIN \leq_K \mathcal Z\times \FIN$. Therefore, $C_0(U_{\mathcal I(\mathcal A)}^*)\simeq c_{0,\mathcal I(\mathcal A)}$ is not Grothendieck. 

    To show that $C(K_\mathcal A)$ is not Grothendieck we can argue in two ways. We can either use the fact that for any $A\in \mathcal A$, the sequence $(n)_{n\in A}$ is convergent to the point $\{A\}\in K_{\mathcal A}$. It is a standard fact that if $K$ has a convergence sequence then $C(K)$ is not Grothendieck.

Alternatively, we can argue starting from the fact that $C_0(U_{\mathcal I(\mathcal A)}^*)$ is not Groth\-en\-dieck. 
    By Theorem \ref{Main Grothendieck} there is an anti-Grothendieck pair $(I_n)_{n\in \omega}\subseteq \mathcal I(\mathcal A)$,  $(\mu_{n})_{n\in \omega}$ of non-negative measures, with $\epsilon>0$. 
    
    For any $n\in \omega$ find $A_{0,n},\dots, A_{i_n,n}\in \mathcal A$ and $F_n\in [\omega]^{<\omega}$ such that $$I_n\subseteq A_{0,n}\cup\dots \cup A_{i_n,n}\cup F_n=X_n.$$ Let $1/a_n=\mu_n(X_n)\ge \mu_n(I_n)>\epsilon$.

Consider the space $E=[\{\chi_A:A\in \mathcal A\cup[\omega]^{<\omega}\}]\subseteq \ell_\infty$, and the operator $T:E\to c_0$ defined by $\displaystyle T(f)(n)=\int_{\omega} f d\mu_n$ for each $n\in\omega$. Since $E\subseteq c_{0,\mathcal I(\mathcal A)}$, this operator has the correct co-domain. Furthermore, for any $n\in \omega$, $T(a_n\chi_{X_n})=e_n$, so this operator cannot be compact.

Therefore, $E$ is not Grothendieck. As explained in the introduction $E\simeq C(K_\mathcal A)$.
    
\end{proof}

\begin{corollary}
    For any $\alpha<\omega_1$, $\FIN^{\alpha}$ is not Grothendieck.
\end{corollary}

\begin{proof}
    The proof goes by induction over $\alpha$. For $\alpha=1$ this is clear. For $\alpha$ a successor ordinal we get that $\FIN^{\alpha}=\FIN\times\FIN^{\alpha'}$, so it is a consequence of Theorem \ref{Frolik sums}. 
    
    For $\alpha$ a limit ordinal consider $f:\omega\to \alpha$ cofinal and strctly increasing. For any $n\in \omega$ find an ultrafilter $\mathcal V_n$ on $X_{f(n)}$ such that $\FIN^{f(n)}\subseteq \mathcal V_{n}^*$. Define an ultrafilter $\mathcal U_n$ on $X_\alpha$ by $A \in \mathcal U_n$ if and only if $$\{x\in X_{f(n)}: (f(n),x)\in A\}\in \mathcal V_n.$$

Notice for any $n\in\omega$, $\{f(n)\}\times X_{f(n)}\in \mathcal U_n\cap \FIN^{\alpha}$. And for any $I\in \FIN^{\alpha}$ there are only finitely many $n<\omega$ such that $I\in \mathcal U_n$. Therefore, $\FIN^{\alpha}$ is not Grothendieck.
    
\end{proof}

These results suggest that the Grothendieck property is upward closed in the Kat\v{e}tov order. Sadly, this is false. 

Recall that $\bigoplus_{n\in\omega} \mathcal I_n=\{A\subseteq \omega^2:\forall n\in \omega\, (A\cap\{n\}\in \mathcal I_n)\}$. Therefore, $\bigoplus_{n\in\omega}\mathcal I_n$ is a subideal of 
$\sum_{\mathcal I}\mathcal I_n$ for any ideal $\mathcal I$. 

\begin{theorem}\label{Frolik sums}
   Let $(\mathcal I_n)$ be a sequence of ideals and $\mathcal I$ an arbitrary ideal. $\mathcal I$ is Grothendieck if and only if $\sum_{\mathcal I}\mathcal I_n$ is Grothendieck.
\end{theorem}

\begin{proof}
 Assume $\mathcal I$ is not Grothendieck. We may find  an anti-Grothendieck pair $(\mu_m)_{m\in \omega}$ and $(I_m)_{m\in \omega}$.

For any $n\in \omega$ find an ultrafilter $\mathcal U_n$ that extends $\mathcal I_n^*$.
 Define for any $m\in \omega$, $\nu_m$ over $\omega^2$ in the following way:

 $$\nu_m(A) =\mu_m(\{n: A\cap (\{n\}\times \omega)\in \mathcal U_n\}).$$

It is clear that this defines a finitely additive measure over $\omega^2$.
Let $J_m=I_m\times \omega$. Then it is clear for any $m\in \omega$, $\nu_m(J_m)=\mu_m(I_m)=m$. Notice that for any $I\in \sum_{\mathcal I}\mathcal I_n$, $I'=\{n:I\cap (\{n\}\times \omega)\in \mathcal U_n\}\in \mathcal I$. Therefore, $$\lim_{m\to \infty} \nu_m(I)=\lim_{m\to \infty} \mu_m(I')=0.$$

Assume that $\sum_{\mathcal I}\mathcal I_n$ is not Grothendieck. Consider an anti-Grothendieck pair witnessed by $\epsilon>0$: $(I_n)_{n\in \omega}$ and $(\mu_n)_{n\in \omega}$ a sequence of nonnegative finitely additive measures. Define for every $n\in \omega$, $A_n=\{m\in \omega: I_n\cap(\{m\}\times \omega)\notin \mathcal J_m\}\in \mathcal I$. Define $I_n'=I_n\cap (A_n\times \omega)\in \mathcal I\times \varnothing$ and $I''_n=I_n\setminus I_n'\in \bigoplus_{m\in\omega} \mathcal I_m$. Notice that for every $n\in \omega$, either $\mu_n(I'_n)>\epsilon/2$ or $\mu_n(I''_n)>\epsilon/2$. Therefore, one of the following two cases must be true.

 \textbf{Case 1} There is $M\in[\omega]^{\omega}$ such that for any $n\in M$, $\mu_n(I'_n)>\epsilon/2$.

 Consider the measure over $\omega$ defined by $\underline{\mu_n}(A)=\mu_n(A\times \omega)$. It is clear that for any $I\in \mathcal I$, $\lim_{n\to \infty}\underline{\mu_n(I)}=\lim_{n\to \infty}{\mu_n(I\times \omega)}=0$.
Furthermore $\underline{\mu_n(A_n)}=\mu_n(A_n\times\omega)\ge \mu_n(I'_n)>\epsilon$.
Therefore, $\mathcal I$ is not Grothendieck.

 \textbf{Case 2} There is $M\in[\omega]^{\omega}$ such that for any $n\in M$, $\mu_n(I''_n)>\epsilon/2$.

We will recursively construct by a sequence $(J_{n_m})_{m\in \omega}\subseteq \bigoplus_{m\in\omega} \mathcal I_m$ such that for any $m\in \omega$, $\mu_{n_m}(J_{n_m})>\epsilon/4$ and $J_{n_{m+1}}\cap( m\times\omega)=\varnothing.$

Consider $n_0=\min M$. Define $J_{n_0}=I''_{n_0}$. Assume that we have built the sequence up to $i-1\in \omega$. The fact that $i\times\omega\in \mathcal I\times \varnothing$ implies that there is $k_i$ such that for any $m\in M$, $m\ge i$, $\mu_{n_m}(i\times\omega)<\epsilon/4$. Define $J_{n_i}=I''_{n_i}\setminus (i\times\omega)$.

Define $J=\bigcup_{n\in \omega} J_{n_m}$. It is easy to see that $J\in \bigoplus_{m\in\omega} \mathcal I_m$. Then for any $m\in \omega$, $\mu_{n_m}(J)>\epsilon/4$. This is a contradiction. Therefore, case 2 is impossible.
 
\end{proof} 

This result has many interesting corollaries.

\begin{corollary}\label{Grothendieck and Rudin-Keisler}
    The set of ideals without the Grothendieck property is Rudin-Keisler (and Kat\v{e}tov) cofinal.
\end{corollary}
\begin{proof}
    Consider $\mathcal I$ an ideal and $f:\omega^2\to\omega$ by $f(n,m)=m$. Then $f$ is a Rudin-Keisler reduction from $\mathcal Z\times \mathcal I$ to $\mathcal I$.
\end{proof}

\begin{corollary}\label{Grothendieck and Katetov}
    The Grothendieck property is not downward closed in the Kat\v{e}tov nor in the Rudin-Kiesler order.
\end{corollary}
\begin{proof}
    Consider $\mathcal U$ an ultrafilter. Since $c_{0,\mathcal U^*}\simeq\ell_\infty$, $\mathcal U^*$ is a Grothendieck ideal. By the previous results $\mathcal U^*\leq_{RK}\mathcal Z\times \mathcal U^*$ and $\mathcal Z\times \mathcal U^*$ is not Grothendieck.
\end{proof}

Using Theorem \ref{Frolik sums}, we can also establish the following corollary:

\begin{corollary}
    Given a sequence $(\mathcal I_n)_{n\in \omega}$ of ideals, $\sum_{\nwd} \mathcal I_n$ and $\sum_{\mathcal H} \mathcal I_n$ are Grothendieck.
\end{corollary}

It is clear that for any to ideals $\mathcal J\leq_{RK}\mathcal I\times\mathcal J$. This is enough to stablish the following corollary:

\begin{corollary}
    The set of ideals with the Grothendieck property is Rudin-Kiesler (and Kat\v{e}tov) cofinal.
\end{corollary}

Recall there are $2^\mathfrak c$ different ideals on $\omega$ and $\mathfrak c$ different analytic ideals.

\begin{corollary}
    There are $2^\mathfrak c$ different  Grothendieck ideals and $\mathfrak c$ different analytic Grothendieck ideals.
\end{corollary}

Concerning the sum and restriction of ideals we have the following:

\begin{theorem}
\begin{enumerate}
    \item Let $(\mathcal I_n)$ be a sequence of ideals. Then, $\bigoplus_{n\in\omega}\mathcal I_n$ is Grothendieck if and only if every $n\in \omega$, $\mathcal I_n$  is Grothendieck . 
    \item If $\mathcal I$ is Grothendieck, then $\mathcal I\restriction X$ is Grothendieck for every $X\in\mathcal I^+$.
\end{enumerate}    
\end{theorem}

\begin{proof}
  (1)  By \cite[Theorem 5.4]{Uzcategui-I-convergence}, $c_{0,\bigoplus_{n\in\omega}\mathcal I_n}\cong\ell_\infty((c_{0,\mathcal I_n})_{n\in\omega})$. By
    \cite[Theorem 5.4.4]{GrothendieckLandscape}, we conclude that $c_{0,\bigoplus_{n\in\omega}\mathcal I_n}$  is Grothendieck if and only if every $c_{0,\mathcal I_n}$  is Grothendieck.
    
    (2) By the proof of \cite[Proposition 5.9]{Uzcategui-complemented}, $c_{0,\mathcal I\restriction X}$ is isometric to a complemented subspace of $c_{0,\mathcal I}$.
    Thus, $c_{0,\mathcal I\restriction X}$ is Grothendieck.
\end{proof}

This last result implies that $\bigoplus_{n\in\omega}\mathcal I_n$ cannot in general be complemented in $c_{0,\sum_{\mathcal I}\mathcal J_n}$, as the latter may be Grothendieck while the former is not.

\section{Nikodym ideals.}

In this section we will analyze the Nikodym property for ideals, we use similar techniques to the ones already presented in the previous section.

We will use \cite[Theorem C]{Sobota-Zuchowski} rephrased as follows:

\begin{theorem}\label{Sobota-Zuchowski}
$\mathcal I$ has the Nikodym property if and only if there is no sequence of disjointly supported non-negative measures $(\mu_n)_{n\in \omega}$ and $(I_n)_{n\in \omega}\subseteq \mathcal I$ such that 
\begin{enumerate}
    \item For any $n\in \omega$, $\mu_{n}(I_n)=\mu(\omega)=n$.
    \item For any $I\in \mathcal I$, $\lim_{n\to \infty}\mu_n(I)=0$.
\end{enumerate}
\end{theorem}

From this statement and Theorem \ref{Main Grothendieck} we get an elementary proof of:

\begin{theorem}\cite[Theorem E]{Sobota-Zuchowski}\label{Nikodym and Grothendieck}
    If $\mathcal I$ is a Grothendieck ideal, then it is a Nikodym ideal.
\end{theorem}

The following results are motivated by the following facts. For a lsco submeasure $\varphi$ such that $\varphi(\omega)=\infty$, if $\mathcal I\subseteq \Exh(\varphi)$, then $\mathcal I$ does not have the Nikodym property. This fact is witnessed by measures of finite support; see \cite[Theorem 4.3]{Zuchowski}. It is easy to see that having a sequence of measures with finite support that witness the failure of the Nikodym property is a property downward closed in the Kat\v{e}tov order; see \cite[Proposition 5.2]{ Zuchowski}. The following is an analog of Theorem \ref{Result on P}, the proof is essentially the same.

\begin{theorem}
    Let $\mathcal I$ be a $P$-ideal. $\mathcal I$ is Nikodym if and only if there are no finitely supported positive measures $(\mu_n)_{n\in \omega}$ on $\mathcal P(\omega)$ such that for any $I\in \mathcal I$, $\lim_{n\in \omega}(I)=0$ and $\sup_{n\in \omega}\mu_n(\omega)=\infty$.
\end{theorem}
\begin{proof}
    If there is a sequence of finitely supported measures $(\mu_n)_{n\in \omega}$ over $\omega$ such that for any $I\in \mathcal I$, $\lim_{n\in \omega}(I)=0$ and $\sup_{n\in \omega}\mu_n(\omega)=\infty$. It is easy to find a subsequence such that, after after re indexing, for any $n\in \omega$, $\mu_n(\supp(\mu_n))\ge n$. Define for any $n\in \omega$, $\nu_n=n\cdot\mu_n/\mu_n({\supp(\mu_n)})$. It is easy to check that this sequence of functions fulfills the properties stated in theorem \ref{Sobota-Zuchowski}. So $\mathcal I$ is not Nikodym.

    Assume $\mathcal I$ is not Nykodym. Consider two sequences $(\mu_n)_{n\in \omega}$, $(I_n)_{n\in\omega}\subseteq\mathcal I$ as in the statement of \ref{Sobota-Zuchowski}. Find $I\in \mathcal I$ such that for any $n\in \omega$, $I_n\setminus I=F_n\in[\omega]^{<\omega}$. Find $m\in \omega$, $m>2$ such that for any $n\ge m$, $\mu_n(I)<1/2$. Define for any $n\in \omega$, $\nu_n=\mu_{n+m}\lvert_{F_{n+m}}$. It is clear that for any $n\in \omega$, $\nu_n(F_n)\ge n.$ Therefore a proper rescaling as in the previous part of the proof is enough to find a sequence as in the statement of the theorem.

\end{proof}

We now present an analog of Theorem \ref{Z kills}.

\begin{theorem}
    Let $\varphi$ be a non-pathological lsco submeasure on $\omega$ such that $\varphi(\omega)=\infty$, and $\mathcal I$ and ideal such that $\mathcal I\leq_K \Exh(\varphi)\times \FIN$, then $\mathcal I$ does not have the Nikodym property. 
\end{theorem}
\begin{proof}
    Consider a Kat\v{e}tov reduction $f:\omega^2\to \omega$ from $\Exh(\varphi)\times \FIN$ to $\mathcal I$. If $\mathcal I$ is not tall then it is clear it is not Nikodym. So assume $\mathcal I$ is tall.
    
    Define $X_n=f(\{n\}\times \omega)$. If $X_n\in \mathcal I$, define $J_n=X_n$, otherwise use tallness to find $J_n\in \mathcal I\cap[X_n]^{\omega}$. For any $n\in \omega$, if $J_n$ is infinite, find an ultrafilter such that $J_n\in\mathcal U_n$ and define $\nu_n=\delta_{\mathcal U_n}$. If $J_n$ is finite, define $F_n=\{i\in J_n: f^{-1}(\{i\}) \text{ is infinite}\}\neq \varnothing$ and $\nu_n=\frac{1}{\abs{F_n}}\sum_{i\in F_n}\delta_i$. 

    Since $\varphi$ is non-pathological find $(\lambda_m)_{m\in \omega}$ such that $\varphi=\sup_{m\in \omega}\lambda_m$. 
    Define $\mu_m=\sum_{n\in \supp (\lambda_n)}\lambda(\{n\})\nu_n$. Notice for any $m\in \omega$, if $I_m\subseteq B$, then $\mu_m(B)=1$.

\,

 \textbf{Claim} $\sup_{n\in \omega}\norm{\mu_m}=\infty$.

 Fix $r\in \mathbb{R}$. Since $\varphi(\omega)=\infty$ we can find $m\in \omega$ and $A\subseteq \omega$ such that $\mu_m(A)>r$. Define $B=\bigcup_{i\in A}J_i$. Then $\mu_m(B)=\sum_{i\in A}\lambda_i(\{i\})\nu_n(B)=\sum_{i\in A}\lambda_i(\{i\})=\lambda(A)>r$.

\,

\textbf{Claim} For any $F\in \mathcal I^*$, $\inf\{\mu_n(F): F\in \mathcal I ^*\}=0$.

Consider $n\in \omega$ and $\epsilon>0$, since $\lambda_n$ is a finite measure, find $k\in \omega$ such that $\lambda_n([k,\infty))<\epsilon$.

Let $F=\omega\setminus \bigcup_{i<k}J_i$. It is clear $F\in \mathcal I^*$. Furthermore,

$$\mu_n(F)=\sum_{i\ge k}\lambda_m(\{i\})\nu_i(F)=\sum_{i\ge k}\lambda_m(\{i\})<\epsilon.$$

\,
    
    \textbf{Claim} For any $I\in \mathcal I$, $\lim_{m\ \to \infty}\mu_m(I)=0$. 

    Pick your favorite $I\in \mathcal I$ and $\epsilon>0$. If $$P_I=\{n\in \omega:\abs{\{m:(n,m)\in f^{-1}(I)\}}= \omega\},$$ then $P_I\in \mathcal \Exh(\varphi)$. Using this fact, find $M\in \omega$ such that for any $m\geq M$, $\varphi(P_I\setminus m)<\epsilon$.

    If $n\in \omega\setminus P_I$, then $I\cap J_n$ is finite. If $J_n$ is infinite, then $\nu_n(I)=0$. If $J_n$ is finite, then $i\notin F_n$, so $\nu_n(I)=0$. Therefore, for $m\ge M$,

    $$\mu_m(I\setminus m)=\sum_{i\in P_I\setminus m}\lambda_m(\{i\})\nu_i(I)+ \sum_{i\in \omega\setminus (P_I\cup m)}\lambda_m(\{i\})\nu_i(I\setminus m)\leq \varphi(P_{I}\setminus m)<\epsilon.$$

    By Theorem \ref{Sobota-Zuchowski}, $\mathcal I$ is not Nikodym.

\end{proof}

As in the previous section, it can also be proven that the Nikodym property behaves well with Frol\'ik sums. Actually the same proof as in Theorem \ref{Frolik sums} works.

\begin{theorem}
   Let $(\mathcal I_n)$ be a sequence of ideals and $\mathcal I$ an arbitrary ideal. Then, $\mathcal I$ is Nikodym if and only if $\sum_{\mathcal I}\mathcal I_n$ is Nikodym.
\end{theorem}

\begin{corollary}
    The set of ideals without the Nikodym property and the set of ideas with the Nikodym property are Rudin-Keisler (and Kat\v{e}tov) cofinal.
\end{corollary}

\begin{corollary}
    The Nikodym property is not downward closed in Rudin-Kiesler nor in the Kat\v{e}tov order.
\end{corollary}

By Theorem \ref{Nikodym and Grothendieck}, Corollaries \ref{Grothendieck and Rudin-Keisler} and \ref{Grothendieck and Katetov} can be deduced from the present corollaries. Furthermore, by \ref{nowhere dense}, \ref{little nowhere dense}, and \ref{Result on P} we can deduce the following corollary.

\begin{corollary}
    $ \nwd$ and $\mathcal H$ are Nikodym ideals.
\end{corollary}

We may wonder if there are any other Borel ideals with the Nykodim property.

\begin{theorem}\label{Independence kills}
    Let $\mathcal I$ be an ideal, $\mathcal H\subseteq \mathcal I$ and $\{I_n:n\in \omega\}\subseteq \mathcal H$ such that for any $n\in \omega$, $\{I_n\}\cup\{\omega\setminus I:I\in \mathcal H\}$ is centered and $\mathcal H$ generated $\mathcal I$. Then $\mathcal I$ is not Nykodym.
\end{theorem}

\begin{proof}
    Find for any $n\in \omega$ an ultrafilter $\mathcal U_n$ such that $$\{I_n\}\cup\{\omega\setminus I:I\in \mathcal H\}\subseteq \mathcal U_n.$$ 

    Define for $n\in \omega$, $\mu_n=n\delta_{\mathcal U_n}$. We claim for any $I\in \mathcal I$, $\lim_{n\to \infty}\mu_n(I)=0$. To show this fix $I\in \mathcal I$. Find $A_0,\dots, A_m\in \mathcal H$ such that $I\subseteq \bigcup_{i\leq m} A_i$. For any $i\leq m$ find $n_i\in \omega$ such that for any $k\ge n_i$, $A_i\notin \mathcal U_n$. Define $n=\max\{n_i:i\leq m\}$, then for every $k\ge n$, $\mu_k(I)=0$. This establishes the claim.
\end{proof}

\begin{corollary}
    Let $\mathcal I$ be an ideal, and $\mathcal H\subseteq \mathcal I$ such that it generates $\mathcal I$ and it is an independent family. Then $\mathcal I$ is not Nikodym.
\end{corollary}
\begin{proof}
    It is enough to pick any countable family $\{I_n:n\in \omega\}\subseteq \mathcal H$. As $\mathcal H$ is independent. for any  $n\in \omega$, $\{I_n\}\cup\{\omega\setminus I:I\in \mathcal H\}$ is centered. Apply the previous theorem.
\end{proof}

We can use this corollary to show that a couple of Borel ideals are also not Nikodym.

\begin{proposition} Solecki's ideal $\mathcal S$, is not Nikodym. 
\end{proposition}

\begin{proof}

Recall that $\mathcal S$ has as its base set: 
     $$N=\{A\subseteq 2^\omega: A \text{ is clopen and } \lambda(A)=1/2\}.$$ 
     
     Where $\lambda$ is the Lebesgue measure on $2^\omega$. For every $x\in 2^\omega$, we may consider $I_x=\{A\in N:x\in A\}$. $\mathcal H=\{I_x:x\in 2^\omega\}$ generates $\mathcal S$. It is easy tho see it is an independent family.
\end{proof}

\begin{proposition}
    The ideal of graphs with finite chromatic number, denoted by $\mathcal G_{fc}$, is not Nikodym.
\end{proposition}
\begin{proof}
    Recall that $\mathcal G_{fc}$  has as its base set $[\omega]^2$. For any $A\subseteq \omega$, define $$A\otimes A^{c}=\{\{n,m\}:n\in A\land m\in A^{c}\}.$$ $\mathcal A=\{A\otimes A^{c}:A\subseteq \omega\}$ generates $\mathcal G_{fc}$. It is easy to see it is an independent family.
\end{proof}

\section{Open Questions.}

We showed that there are many non-isomorphic Grothendieck ideals, however, many of them give rise to isomorphic Grothendieck spaces. Thus we ask:

\begin{enumerate}

\item How many non isomorphic Grothendieck spaces of the form $c_{0,\mathcal I}$ are there? 
\smallskip
\noindent
We also find the last question restricted to definable ideals interesting:
\smallskip
\item How many non isomorphic Grothendieck spaces of the form $c_{0,\mathcal I}$ are there for $\mathcal I$ analytic/Borel? 

\end{enumerate}

\bibliographystyle{plain}
\bibliography{references}

\end{document}